\documentclass{ajmart}

\issueinfo{2}{4}{December}{2008}

\copyrightinfo{2008}{Aulona Press  \emph{(Albanian J. Math.)}}

\def\issn{{\sc ISSN} 1930-1235: }
\def\issueyear{2008}

\PII{\issn  (\issueyear )}

\pagespan{307}{318}

\usepackage{amsmath, amsbsy,amsfonts,amsopn,amstext, euscript, amssymb,amsbsy,amsfonts,amsthm,latexsym,amsopn,amstext, amsxtra,euscript,amscd}

\usepackage{graphicx}
\input xy \xyoption{all}

\newtheorem{theorem}{Theorem}
\newtheorem{proposition}{Proposition}
\newtheorem{lemma}{Lemma}
\newtheorem{remark}{Remark}

\theoremstyle{definition}

\theoremstyle{remark}

\title{Degree 4 Coverings of Elliptic Curves by Genus 2 Curves}

\def\Q{\mathbb Q}
\def\bP{\mathbb P}

\def\C{\mathcal C}

\def\H{\mathcal H}
\def\M{\mathcal M}
\def\L{\mathcal L}

\def\lar{\longrightarrow}

\def\s{\sigma}

\def\H{\mathcal H}

\def\D{\Delta}

\def\Aut{\mbox{Aut}}

\def\({\left(}
\def\){\right)}

\def\<{\langle}
\def\>{\rangle}

\def\s{\sigma}

\def\<{\langle}
\def\>{\rangle}
\def\_u{{\mathfrak u}}

\begin{document}

\maketitle

\begin{center}
{\sc T. Shaska, G.S. Wijesiri} \\
\vspace*{1ex}
\emph{Department of Mathematics\\
Oakland University\\
Rochester, MI, 48309-4485.}
\end{center}

\vspace*{3ex}

\begin{center}
{\sc S. Wolf} \\
\vspace*{1ex}
\emph{Department of Mathematics\\
Cornell University\\
Ithaca, NY 14853-4201. }
\end{center}

\vspace*{3ex}

\begin{center}
{\sc L. Woodland } \\
\vspace*{1ex}
\emph{Department of Mathematics $\&$ Computer Science\\
Westminster College\\
501 Westminster Avenue\\
Fulton MO 65251-1299.}
\end{center}

\vspace*{3ex}

\begin{abstract}
Genus two curves covering elliptic curves have been the object of study of many articles. For a fixed degree $n$ the
subloci of the moduli space $\M_2$ of curves having a degree $n$ elliptic subcover has been computed for $n=3, 5$ and discussed in detail for $n$ odd; see \cite{Sh1, SV2, Fr, FK}. When the degree of the cover is even the case in general has been treated in \cite{PRS}.  In this paper we compute the sublocus of $\M_2$ of curves having a degree $4$ elliptic subcover.
\end{abstract}

\section{Introduction}
Let $\psi: C \to E$ be a degree $n$ covering of an elliptic curve $E$ by a genus two curve $C$. Let $\pi_C: C \lar \bP^1$ and $\pi_E:E \lar \bP^1$ be
the natural degree 2 projections. There is $\phi:\bP^1 \lar \bP^1$ such that the diagram commutes.
\begin{equation}
\begin{matrix}
C & \buildrel{\pi_C}\over\lar & \bP^1\\
\psi \downarrow & & \downarrow \phi \\
E & \buildrel{\pi_E}\over\lar & \bP^1
\end{matrix}
\end{equation}
The ramification of induced coverings $\phi:\bP^1 \lar \bP^1$ can be determined in detail; see \cite{PRS} for details. Let $\s$ denote the fixed
ramification of $\phi:\bP^1 \lar \bP^1$. The Hurwitz space of such covers is denoted by $\H (\s)$. For each covering $\phi:\bP^1 \lar \bP^1$
(up to equivalence) there is a unique genus two curve $C$ (up to isomorphism). Hence, we have a map
\begin{equation}
\begin{split}
\Phi: \, \, \H(\s) \to \M_2 \\\
 [\phi] \to [C].\\
\end{split}
\end{equation}
We denote by $\L_n (\s) $ the image of $\H (\s)$ under this map. The main goal of this paper is to study $\L_4 (\s)$.

\section{Preliminaries}
Most of the material of this section can be found in \cite{g2}. Let $C$ and $E$ be curves of genus 2 and 1, respectively. Both are smooth, projective curves defined over $k$, $char(k)=0$. Let $\psi: C \longrightarrow E$ be a covering of degree $n$. From the Riemann-Hurwitz formula, $\sum_{P \in C}\, (e_{\psi}\,(P) -1)=2$ where $e_{\psi}(P)$ is the ramification index of points $P \in C$, under $\psi$. Thus, we have two points of ramification index 2 or one point of ramification index 3. The two points of
ramification index 2 can be in the same fiber or in different fibers. Therefore, we have the following cases of
the covering $\psi$:\\

\textbf{Case I:} There are $P_1$, $P_2 \in C$, such that $e_{\psi}({P_1})=e_{\psi}({P_2})=2$,
$\psi(P_1) \neq \psi(P_2)$, and $\forall P \in C\setminus \{P_1,P_2\}$, $e_{\psi}(P)=1$.

\textbf{Case II:} There are $P_1$, $P_2 \in C$, such that $e_{\psi}({P_1})=e_{\psi}({P_2})=2$, $\psi(P_1) =
\psi(P_2)$, and $\forall P \in C\setminus \{P_1,P_2\}$, $e_{\psi}(P)=1$.

\textbf{Case III:} There is $P_1 \in C$ such that $e_{\psi}(P_1)=3$, and $ \forall P \in C \setminus \{P_1\}$,
$e_{\psi}(P)=1$.\\

\noindent In case I (resp. II, III) the cover $\psi$ has 2 (resp. 1) branch points in E.

Denote the hyperelliptic involution of $C$ by $w$. We choose $\mathcal O$ in E such that $w$ restricted to$E$ is the hyperelliptic involution on $E$. We denote the restriction of $w$ on $E$ by $v$, $v(P)=-P$. Thus, $\psi \circ w=v \circ \psi$. E[2] denotes the group of 2-torsion points of the elliptic curve E, which are the points fixed by $v$. The proof of the following two lemmas is straightforward and will be omitted.

\begin{lemma} \label{lem_1}
a) If $Q \in E$, then $\forall P \in \psi^{-1}(Q)$, $w(P) \in \psi^{-1}(-Q)$.

b) For all $P\in C$, $e_\psi(P)=e_\psi\,({w(P)})$.
\end{lemma}

Let $W$ be the set of points in C fixed by $w$. Every curve of genus 2 is given, up to isomorphism, by a binary sextic, so there are 6 points fixed by the hyperelliptic involution $w$, namely the Weierstrass points of $C$. The
following lemma determines the distribution of the Weierstrass points in fibers of 2-torsion points.

\begin{lemma}\label{lem2} The following hold:
\begin{enumerate}
\item $\psi(W)\subset E[2]$
\item If $n$ is an even number then for all $Q\in E[2]$, \#$(\psi^{-1}(Q) \cap W)=0 \mod (2)$
\end{enumerate}
\end{lemma}

Let $\pi_C: C \lar \bP^1$ and $\pi_E:E \lar \bP^1$ be the natural degree 2 projections. The hyperelliptic
involution permutes the points in the fibers of $\pi_C$ and $\pi_E$. The ramified points of $\pi_C$, $\pi_E$
are respectively points in $W$ and $E[2]$ and their ramification index is 2. There is $\phi:\bP^1 \lar \bP^1$
such that the diagram commutes.
\begin{equation}
\begin{matrix}
C & \buildrel{\pi_C}\over\lar & \bP^1\\
\psi \downarrow & & \downarrow \phi \\
E & \buildrel{\pi_E}\over\lar & \bP^1
\end{matrix}
\end{equation}
Next, we will determine the ramification of induced coverings $\phi:\bP^1 \lar \bP^1$. First we fix some
notation. For a given branch point we will denote the ramification of points in its fiber as follows. Any
point $P$ of ramification index $m$ is denoted by $(m)$. If there are $k$ such points then we write $(m)^k$.
We omit writing symbols for unramified points, in other words $(1)^k$ will not be written. Ramification data
between two branch points will be separated by commas. We denote by $\pi_E (E[2])=\{q_1, \dots , q_4\}$ and
$\pi_C(W)=\{w_1, \dots ,w_6\}$.

Let us assume now that $deg(\psi)=n$ is an even number. Then the generic case for $\psi: C \lar E$ induce the following three cases for $\phi: \bP^1 \lar \bP^1$:

\begin{description}
\item[I] $ \left ( (2)^\frac {n-2} 2 , (2)^\frac {n-2} 2 , (2)^\frac {n-2} 2 , (2)^ \frac {n}
2 , (2) \right ) $
\item[II] $ \left ( (2)^\frac {n-4} 2 , (2)^\frac {n-2} 2 , (2)^\frac {n} 2 , (2)^ \frac {n} 2 ,
(2) \right ) $
\item[III] $ \left ( (2)^\frac {n-6} 2 , (2)^\frac {n} 2 , (2)^\frac {n} 2 , (2)^ \frac {n} 2 , (2)
\right ) $
\end{description}
Each of the above cases has the following degenerations (two of the branch points collapse to one)

\begin{description}
\item[I]
\begin{enumerate}
\item $\left ( (2)^\frac {n} 2 , (2)^\frac {n-2} 2 ,
 (2)^\frac {n-2} 2 , (2)^ \frac {n} 2 \right )
$ \item $\left ( (2)^\frac {n-2} 2 , (2)^\frac {n-2} 2 , (4) (2)^\frac {n-6} 2 , (2)^ \frac {n} 2 \right ) $
\item $\left ( (2)^\frac {n-2} 2 , (2)^\frac {n-2} 2 ,
 (2)^\frac {n-2} 2 , (4) (2)^ \frac {n-4} 2 \right )
$ \item $\left ( (3) (2)^\frac {n-4} 2 , (2)^\frac {n-2} 2 ,
 (2)^\frac {n-2} 2 , (2)^ \frac {n} 2 \right )
$
\end{enumerate}
\item[II]
\begin{enumerate}
\item $\left ( (2)^\frac {n-2} 2 , (2)^\frac {n-2} 2 , (2)^\frac {n} 2 , (2)^ \frac {n} 2
\right ) $
\item $\left ( (2)^\frac {n-4} 2 , (2)^\frac {n} 2 , (2)^\frac {n} 2 , (2)^ \frac {n} 2 \right)$
\item $\left ((4) (2)^\frac {n-8} 2, (2)^\frac {n-2} 2 , (2)^\frac {n} 2, (2)^ \frac {n}2\right )$
\item $\left ( (2)^\frac {n-4} 2 , (4) (2)^\frac {n-6} 2 , (2)^\frac {n} 2 , (2)^ \frac {n} 2
\right ) $ \item $\left ( (2)^\frac {n-4} 2 , (2)^\frac {n-2} 2 , (2)^\frac {n-4} 2 , (2)^ \frac {n} 2 \right
) $ \item $\left ((3) (2)^\frac {n-6} 2 , (2)^\frac {n-2} 2 , (4) (2)^\frac {n} 2 , (2)^ \frac {n} 2 \right )
$
\item $\left ( (2)^\frac {n-4} 2 , (3) (2)^\frac {n-4} 2 ,
 (2)^\frac {n} 2 , (2)^ \frac {n} 2 \right )
$
\end{enumerate}
\item[III]
\begin{enumerate}
\item $\left ( (2)^\frac {n-4} 2 , (2)^\frac {n} 2 ,
 (2)^\frac {n} 2 , (4) (2)^ \frac {n} 2 \right )
$ \item $\left ( (2)^\frac {n-6} 2 , (4) (2)^\frac {n-4} 2 ,
 (2)^\frac {n} 2 , (2)^ \frac {n} 2 \right )
$ \item $\left ( (2)^\frac {n} 2 , (2)^\frac {n} 2 ,
 (2)^\frac {n} 2 , (4) (2)^ \frac {n-10} 2 \right )
$ \item $\left ( (3) (2)^\frac {n-8} 2 , (2)^\frac {n} 2 ,
 (2)^\frac {n} 2 , (2)^ \frac {n} 2 \right )
$
\end{enumerate}
\end{description}

For details see \cite{PRS}.

\section{Degree 4 case}

In this section we focus on the case $\deg (\phi) = 4$. The goal is to determine all ramifications $\s$ and explicitly compute $\L_4 (\s)$. \\

There is one generic case and one degenerate case in which the ramification of $\deg (\phi) = 4$ applies, as given by
the above possible ramification structures:
\begin{enumerate}
\item[i)] $(2,2,2,2^2,2)$ (generic)
\item[ii)] $(2,2,2,4)$ (degenerate)
\end{enumerate}

\section{Computing the locus $\mathcal L_4$ in $\mathcal M_2$}

\subsection{Non-degenerate case}
Let $\psi : C \longrightarrow E $ be a covering of degree 4, where $C$ is a genus 2 curve and $E$ is an elliptic curve.
Let $\phi$ be the Frey-Kani covering with $deg(\phi)=4$ such that $\phi(1)=0,$ $\phi( \infty)=\infty,$ $\phi(p)=\infty$
 and the roots of $f(x) = x^2+ax+b$ be in the fiber of $0.$ In the following figure, bullets
(resp., circles) represent places of ramification index 2 (resp., 1).

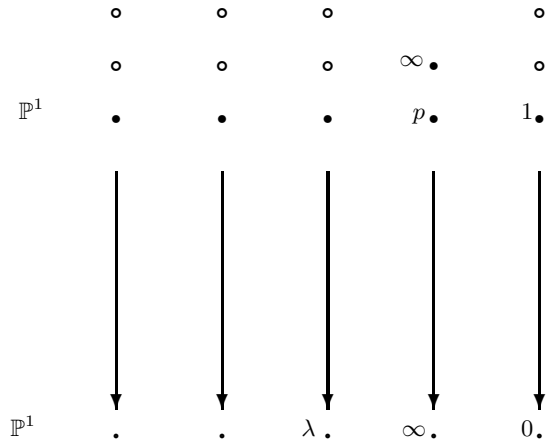
\begin{figure}[h]
\begin{center}
\begin{small}
\begin{picture}(250,200) (-10,0)
\thicklines

\put(10,100){\vector(0,-1){90}}
\put(50,100){\vector(0,-1){90}}

\put(90,100){\vector(0,-1){90}}
\put(130,100){\vector(0,-1){90}}
\put(170,100){\vector(0,-1){90}}

\put(10,0){\circle{.2}}
\put(50,0){\circle{.2}}
\put(90,0){\circle{.2}}
\put(130,0){\circle{.2}}
\put(170,0){\circle{.2}}

\put(10,120){\circle*{2.5}}
\put(50,120){\circle*{2.5}}
\put(90,120){\circle*{2.5}}
\put(130,120){\circle*{2.5}}
\put(170,120){\circle*{2.5}}

\put(10,140){\circle{2.5}}
\put(50,140){\circle{2.5}}
\put(90,140){\circle{2.5}}
\put(130,140){\circle*{2.5}}
\put(170,140){\circle{
2.5}}

\put(10,160){\circle{2.5}}
\put(50,160){\circle{2.5}}
\put(90,160){\circle{2.5}}
\put(170,160){\circle{2.5}}

\put(118,0){\small{$\infty$}} \put(160,0){\small{ $0$}} \put(80,0){\small{$\lambda$}}

\put(119,120){\small{ $p$}} \put(114,140){\small{ $\infty$}} \put(160,120){\small{ $1$}}

\put(-30,120){\small{ $\mathbb{P}^1$}}
\put(-30,0){\small{$\mathbb{P}^1$}}

\end{picture}
\caption{Degree 4 covering for generic case} 
\end{small}
\end{center}
\end{figure}
\noindent Then the cover can be given by \[ \phi(x)= \frac{k(x-1)^2(x^2+b)}{(x-p)^2}.\] Let $\lambda$ be a 2-torsion
point of $E.$ To find $\lambda,$ we solve
  \begin{equation}\label{lambda_phi}
  \phi(x) -\lambda =0.\end{equation}
According to this ramification we should have 3 solutions for $\lambda,$ say $\lambda_1, \lambda_2, \lambda_3.$ The
discriminant of the Eq.~\eqref{lambda_phi} gives branch points for the points with ramification index 2. So we have the
following relation for $\lambda,$ with $p \neq 1.$
\begin{equation}\label{lambda}
\begin{split}
& \left( -b-{p}^{2} \right) {\lambda}^{3}+ \left( 2\,k{p}^{2}-18\,kbp+
16\,k{p}^{4}-16\,k{p}^{3}+3\,k{b}^{2}+3\,kb+20\,kb{p}^{2} \right) { \lambda}^{2}\\
&+ ( -3\,{k}^{2}b+21\,{k}^{2}{b}^{2}-36\,{k}^{2}{b}^{2
}p-3\,{k}^{2}{b}^{3}-20\,{k}^{2}b{p}^{2}+8\,{k}^{2}{b}^{2}{p}^{2}+18\, {k}^{2}bp\\
&-{k}^{2}{p}^{2} ) \lambda+{k}^{3}b+{k}^{3}{b}^{4}+3\,{k }^{3}{b}^{2}+3\,{k}^{3}{b}^{3}=0.
\end{split}
\end{equation}
Using Eq.\eqref{lambda_phi} and Eq.\eqref{lambda} we find the degree 12 equation with 2 factors. One of them with
degree 6 corresponds to the equation of genus 2 curve and the other corresponds to the double roots in the fiber of
$\lambda_1,$ $\lambda_2$ and $\lambda_3.$

The equation of genus 2 curve can be written as follows:
\[
C: y^2 = a_6 x^6+ a_5 x^5+ a_4 x^4 + a_3 x^3+ a_2 x^2+ a_1 x +a_0
\]
where
\[
\begin{split}
a_6 =\, & {p}^{2}+b \\
a_5 =\, & 4\,{p}^{3}-6\,{p}^{2}+4\,pb-6\,b \\
a_4 = \,& -4\,{p}^{4}-10\,{p}^{3}+ \left( -5\,b+13 \right) {p}^{2}-8\,pb+12\,b \\
a_3 = \,& 12\,{p}^{4}+ \left( 4+6\,b \right) {p}^{3}+ \left( -12+12\,b \right) { p}^{2}+ \left( 8\,{b}^{2}-6\,b
\right) p-8\,b-8\,{b}^{2} \\
a_2 = \,& \left( -11-4\,b \right) {p}^{4}+ \left( -20\,b+6 \right) {p}^{3}+
 \left( 4+13\,b-12\,{b}^{2} \right) {p}^{2}+10\,pb+12\,{b}^{2} \\
a_1 =\, & \left( 14\,b+2 \right) {p}^{4}+ \left( 6\,{b}^{2}-4+4\,b \right) {p}^ {3}+ \left( -24\,b+6\,{b}^{2} \right)
{p}^{2}+ \left( -6\,{b}^{2}+4\,b
 \right) p-6\,{b}^{2}
 \\
a_0 =\, &\left( -{b}^{2}+1-11\,b \right) {p}^{4}+ \left( 14\,b-2\,{b}^{2}
 \right) {p}^{3}-2\,b{p}^{2}+2\,{b}^{2}p+{b}^{2}.
\end{split}
\]
Notice that we write the equation of genus 2 curve in terms of only 2 unknowns. We denote the Igusa invariants of $C$
by $J_2, J_4, J_6$, and $J_{10}$. The absolute invariants of $C$ are given in terms of these classical invariants:
\[ i_1 = 144 \frac{J_4}{J_2^2}, \quad i_2 = -1728 \frac{J_2 J_4 - 3 J_6}{J_2^3}, \quad i_3 = 486 \frac{J_{10}}{J_2^5}. \]

Two genus 2 curves with $J_2 \neq 0$ are isomorphic if and only if they have the same absolute invariants. Notice that
these invariants of our genus 2 curve are polynomials in $p$ and $b$. By using a computational symbolic package (as
Maple) we eliminate $p$ and $b$ to determine the equation for the non-degenerate locus $\L_4.$ The result is very long.
We don't display it here.

\section{Degenerate Case}
Notice that only one degenerate case can occur when $n=4:$ $(2,2,2,4).$ In this case one of the Weierstrass points has
ramification index 3, so the cover is totally ramified at this point.

Let the branch points be 0, 1, $\lambda$, and $\infty$, where $\infty$ corresponds to the element of index 4. Then,
above the fibers of 0, 1, $\lambda$ lie two Weierstrass points. The two Weierstrass points above $0$ can be written as
the roots of a quadratic polynomial $x^2 + ax + b$; above $1$, they are the roots of $x^2 +px + q$; and above
$\lambda$, they are the roots of $x^2 + sx + t$. This gives us an equation for the genus 2 curve $C$:
\[ C: y^2 = (x^2 + ax + b)(x^2 +px + q)(x^2 +sx + t). \]
The four branch points of the cover $\phi$ are the 2-torsion points $E[2]$ of the elliptic curve $E$, allowing us to write the elliptic subcover as

\[ E: y^2 = x(x - 1)(x - \lambda). \]
The cover $\phi: \bP^1 \to \bP^1$ is Frey-Kani covering and is given by
\[ \phi(x)=c x^2(x^2 + ax + b). \]
Using $\phi(1)= 1$, we get $c = \frac{1}{1 + a + b}$. Then,
\[ \phi(x)- 1 =c(x-1)^2(x^2 + px+q).\]
This implies that $\phi'(1)= 0$, so we get $c(4 + 3a +2b) = 0$. Since $c$ cannot be $0$, we must have $4 + 3a + 2b = 0$, which implies $ a = \frac{-2(b +2)}{3}$. Combining this with our equation for $c$, we get $ c = \frac{3}{b-1}$.

Now, since $\phi(x)-1 - c (x-1)^2 (x^2 + px + q) = 0$, we want all of the coefficients of this polynomial to be identically 0; thus
\[ p = \frac{2(1 - b)}{3}, q = \frac{1 - b}{3}. \]
Finally, we consider the fiber above $\lambda$. We write
\[ \phi(x)- \lambda = c (x - r)^2(x^2 + sx + t). \]
Similar to above, we set the coefficients of the polynomial to 0 to get:
\[ \lambda = \frac{b^3(4 - b)}{16(b - 1)}, \ \ r = \frac b 2, \ \ s = \frac{b-4}{3}, \ \ t = \frac{b(b-4)}{12}. \]

Hence we have $C$ and $E$ with equations:
\begin{small}
\begin{equation} \label{genus2Degenerate}
\begin{split}
 C:\quad  y^2 & = \left(\frac{1-b}{3}+\frac{2}{3} (1-b) x+x^2\right) \left(\frac{1}{12} (b-4) b+\frac{1}{3} (b-4)
 x+x^2\right)\\
 & \left(b-\frac{2}{3} (b+2) x+x^2\right) \\
 E:\quad  v^2 & = u (u-1) \left(u - \frac{b^3(4 - b)}{16(b - 1)}\right)\\
\end{split}
 \end{equation}
\end{small}
where the corresponding discriminants of the right sides must be non-zero. Hence,
\begin{align}
 \D_C:& = b(b-4)(b-2)(b-1) (2+b) \neq 0 \\
 \D_E:& = \frac{(b-4)^2 (b-2)^6 b^6 (b+2)^2}{65536 (b-1)^4} \neq 0.
\end{align}


From here on, we consider the additional restriction on $b$ that it does not solve $J_2 = 0$, that is,
\begin{equation} \label{J2}
J_2 = -\frac{5}{486}(256 - 384 b - 4908 b^2 + 5068 b^3 - 1227 b^4 - 24 b^5 + 4 b^6) \neq 0.
\end{equation}
The case when $J_2=0$ is considered separately. We can eliminate $b$ from this system of equations by taking the
numerators of $ i_j - i_j(b)$ and setting them equal to 0, where $i_j$ are absolute invariants of genus 2 curve.

Thus, we have 3 polynomials in $b, i_1, i_2, i_3$. We eliminate $b$ using the method of resultants and get the following:
\begin{small}
\begin{equation}
\begin{split}
3652054494822999 - 312800728170302145 i_1 - 247728254774362875 i_1^2 \\
+ 3039113062253125 i_1^3 -522534367747902600 i_2 - 28017734537115000 i_1 i_2 \\
- 238234372300000 i_2^2  & =  0\\
\end{split}\label{degenerate1}
\end{equation}
\end{small}
and the other equation
\begin{small}
\begin{equation}
\begin{split}
 1158391804615233525 i_1 - 17653298856896250 i_1^2 + 100894442906250 i_1^3  \\
- 256292578125 i_1^4 + 244140625 i_1^5 - 323890167989102732668800000i_3   \\
 - 14879672225288904960000000 i_1 i_3 - 40609431102258000000000 i_1^2 i_3  \\
-16677181699666569 +  347405361918358396861440000000000 i_3^2 & =  0 \\
\end{split}\label{degenerate2}
\end{equation}
\end{small}


These equations determine the degenerate locus $\L_4^\prime$ when $J_2 \neq 0.$

When $J_2 = 0$, we must resort to the $a$-invariants of the genus 2 curve. These invariants are defined as
\[ a_1 = \frac{J_4 J_6}{J_{10}}, \qquad a_2 = \frac{J_{10} J_6}{J_4^4}. \]
Two genus 2 curves with $J_2 = 0$ are isomorphic iff their $a$-invariants are equal. For our genus 2 curve,
\begin{equation*}
\begin{split}
J_4 & = \frac 1 {5184 } \left(
 65536-196608 b-307200 b^2+1218560 b^3-834288 b^4-294432 b^5 \right. \\
 & \left.  +456600 b^6 -73608 b^7-52143 b^8+19040 b^9-1200 b^{10}-192 b^{11}+16 b^{12} \right)
\end{split}
\end{equation*}

It can be guarantee that $J_4$ and $J_2$ are not simultaneously $0$ because the resultant of these two polynomials in
$b$ is
\[\frac{11784978051522395707646672896000000000000}{42391158275216203514294433201}, \]
so there are no more subcases. We want to eliminate $b$ from the set of equations:
\begin{eqnarray*}
J_2 &=& 0\\
a_1 - a_1(b) &=& 0\\
a_2 - a_2(b) &=& 0.\end{eqnarray*}
Similar to what we did above with the $i$-invariants, we take resultants of
combinations of these and set them equal to $0$. Doing so tells us
\begin{equation*}
\begin{split}
& 20a_1-55476394831 = 0 \\
& 1022825924657928a_2-522665 = 0.
\end{split}
\end{equation*}
So in other words, if $C$ is a genus 2 curve with a degree 4 elliptic subcover with $J_2 = 0$, then
\[ a_1 = \frac{55476394831}{20}, \quad a_2 = \frac{522665}{1022825924657928}. \]

So up to isomorphism, this is the only genus 2 curve with degree 4 elliptic subcover with $J_2 = 0$. In this case the
equation of the genus 2 curve is given by Eq.\eqref{genus2Degenerate}, where $b$ is given by the following:
\begin{equation}\label{bforJ2zero}
b = \frac{2 \alpha + \sqrt{429 \alpha^2 + 60123 \alpha + \beta}}{2\alpha}
\end{equation}
with $\alpha = \sqrt [3]{2837051+9408\,i\sqrt {5}}$ and $\beta =8511153+28224\,i\sqrt {5}.$
We summarize the above results in the following theorem.
\begin{theorem}
Let $C$ be a genus 2 curve with a degree 4 degenerate elliptic subcover. Then $C$ is isomorphic to the curve given by
Eq.\eqref{genus2Degenerate} where $b$ satisfies Eq.\eqref{bforJ2zero} or its absolute invariants satisfy
Eq.~\eqref{degenerate1} and Eq.~\eqref{degenerate2}.
\end{theorem}

\begin{remark}
The genus 2 curve, when $J_2=0,$ is not defined over the rational.
\end{remark}

\begin{remark}
When the genus 2 curve has non zero $J_2$ invariant the $j$ invariant of the elliptic curve satisfies the following
equation:
\begin{small}
\[
\begin{split}
0=& ( 2621440000000000\,{{\it J_4}}^{4}-14332985344000000\,{{\it J_2}}
^{2}{{\it J_4}}^{3}-15871355368243200\,{{\it J_2}}^{6}{\it J_4}\\
&+ 1586874322944\,{{\it J_2}}^{8}+26122821304320000\,{{\it J_2}}^{4}{{\it J_4}}^{2} ) {j}^{2}+
(-2535107603331605760\,{{\it J_2}}^{8}
\\
&+25102192337335536076800\,{{\it J_2}}^{6}{\it J_4}- 164781024264192000000000\,{{\it
J_4}}^{4}\\
&+90675809529498685440000\,{{ \it J_2}}^{4}{{\it J_4}}^{2} -363163522083397632000000\,{{\it J_2}} ^{2}{{\it J_4}}^{3} )
j\\
&+2589491458659766450406400000000\,{{\it J_4}}^{4}- 203482361042468209670400000000\,{{\it J_2}}^{2}{{\it J_4}}^{3}\\
&+ 39862710766802552045625\,{{\it J_2}}^{8}-19433806326190741141800000\,{{ \it J_2}}^{6}{\it
J_4}\\
&+3259543004362746907416000000\,{{\it J_2}}^{4}{{ \it J_4}}^{2}.
\end{split}
\]
\end{small}

\end{remark}

\subsection{Genus 2 curves with degree 4 elliptic subcovers and extra automorphisms in the degenerate locus of $\L_4$}
In any characteristic different from 2, the automorphism group Aut(C) is isomorphic to one of the groups : $C_2,$
$C_{10},$ $V_4,$ $D_8,$ $D_{12},$ $C_3 \rtimes D_8,$ $GF_2(3),$ or $2^+ S_5;$ See \cite{SV1} for the description of
each group. We have the following lemma.
\begin{lemma} \label{lem3}
\begin{enumerate}

\item[(a)] The locus $\L_2$ of genus 2 curves $C$ which have a degree 2 elliptic subcover is a closed subvariety of $\M_2$. The equation
of $\L_2$ is given by
\begin{equation}
\begin{split}\label{eq_L2_J}
0& = 8748J_{10}J_2^4J_6^2- 507384000J_{10}^2J_4^2J_2-19245600J_{10}^2J_4J_2^3
-592272J_{10}J_4^4J_2^2 \\
& +77436J_{10}J_4^3J_2^4 -3499200J_{10}J_2J_6^3+4743360J_{10}J_4^3J_2J_6
-870912J_{10}J_4^2J_2^3J_6 \\
& +3090960J_{10}J_4J_2^2J_6^2 -78J_2^5J_4^5-125971200000J_{10}^3-81J_2^3J_6^4+1332J_2^4J_4^4J_6 \\
& +384J_4^6J_6+41472J_{10}J_4^5+159J_4^6J_2^3 -236196J_{10}^2J_2^5-80J_4^7J_2
 -47952J_2J_4J_6^4\\
& +104976000J_{10}^2J_2^2J_6-1728J_4^5J_2^2J_6+6048J_4^4J_2J_6^2
-9331200J_{10}J_4^2J_6^2 -J_2^7J_4^4\\
& +12J_2^6J_4^3J_6+29376J_2^2J_4^2J_6^3-8910J_2^3J_4^3J_6^2-2099520000J_{10}^2J_4J_6
+31104J_6^5\\
& -6912J_4^3J_6^34-5832J_{10}J_2^5J_4J_6 -54J_2^5J_4^2J_6^2 +108J_2^4J_4J_6^3 +972J_{10}J_2^6J_4^2 .
\end{split}
\end{equation}

\item[(b)] The locus $\M_2 ( D_8)$  of genus 2 curves $C$ with $\Aut(C) \equiv D_8$ is given by the equation of $\L_2$ and
\begin{equation}
0 = 1706 J_4^2 J_2^2 + 2560 J_4^3 + 27J_4 J_2^4 - 81 J_2^3 J_6 - 14880 J_2 J_4 J_6 + 28800 J_6^2. \label{J8}
\end{equation}

\item[(c)] The locus $\M_2 ( D_{12})$  of genus 2 curves $C$ with $\Aut(C) \equiv D_{12}$ is
\begin{align}
0 &= -J_4 J_2^4 + 12 J_2^3 J_6 - 52 J_4^2 J_2^2 + 80 J_4^3 + 960 J_2 J_4 J_6 - 3600 J_6^2 \label{D12_1}\\
0 &= -864J_{10} J_2^5 + 3456000 J_{10} J_4^2 J_2 - 43200J_{10} J_4 J_2^3 - 2332800000 J_{10}^2 \label{D12_2} \\
&\qquad - J_4^2 J_2^6 - 768J_4^4 J_2^2 + 48 J_4^3 J_2^4 + 4096J_4^5. \nonumber
\end{align}
\end{enumerate}
\end{lemma}

We will refer to the locus of genus 2 curves C with $\Aut(C) \equiv D_{12}$ (resp., $\Aut(C) \equiv D_8$) as the $D_{12}$-locus (resp., $D_8$-locus).

Equations (\ref{degenerate1}), (\ref{degenerate2}), and (\ref{eq_L2_J}) determine a system of 3 equations in the 3
$i$-invariants. The set of possible solutions to this system contains 20 rational points and 8 irrational or complex
points (there may be more possible solutions, but finding them involves the difficult task of solving a degree 15 or
higher polynomial). Among the 20 rational solutions, there are four rational points which actually solve the system.
\begin{align*}
 (i_1, i_2, i_3) &= \left( \frac{102789} {12005}, \frac{-73594737}{2941225}, \frac{531441}{28247524900000} \right) \\
 (i_1, i_2, i_3) &= \left(\frac{66357}{9245},\frac{-892323}{46225},\frac{7776}{459401384375}\right) \\
 (i_1, i_2, i_3) &= \left(\frac{235629}{1156805}, \frac{-28488591}{214008925},\frac{53747712}{80459143207503125}\right)\\
 (i_1, i_2, i_3) &=
 \left(\frac{1078818669}{383775605},\frac{-77466710644803}{16811290377025},\frac{1356226634181762}{161294078381836186878125}\right).
\end{align*}
Of these four points, only the first one lies on the $D_{12}$-locus, and none lie on the $D_8$-locus, so the other
three curves have automorphism groups isomorphic to $V_4$ (See Remark 3 for their equations). We have the following
proposition.

\begin{proposition}
There is exactly one genus 2 curve $C$ defined over $\Q$ (up to $\C$-isomorphism) with a degree 4 elliptic subcover  which
has an automorphism group  $D_{12}$ namely the curve   \[C=100X^6+100X^3+27\]
and no such curves with automorphism group $D_8$.
\end{proposition}

\begin{proof}

From above discussion there is exactly one rational point which lies on the $D_{12}$-locus and three rational points
which lies on the $V_4$-locus. Furthermore we have the fact that $\Aut(C) \equiv D_{12}$ if and only if $C$ is
isomorphic to the curve given by $Y^2 = X^6 + X^3 + t$ for some $t \in k;$ see \cite{Sh6} for more details.

Suppose the equation of the $D_{12}$ case is $Y^2 = X^6 + X^3 + t.$ We want to find $t.$ We can calculate the
$i$-invariants in terms of $t$ accordingly, so we get a system of equations, $i_j - i_j(t)= 0$ for $j \in \{1,2,3\}$.
Those equations simplify to the following:
\begin{align}
0 &= 1600 i_1 t^2 - 80 i_1 t + i_1 - 6480 t^2 - 1296 t \nonumber\\
0 &=64000 i_2 t^3 - 4800 i_2 t^2 + 120 i_2 t- i_2 + 233280 t^3 + 303264 t^2 - 11664t \nonumber\\
0 &=1638400000i_3 t^5 - 204800000i_3 t^4 + 10240000i_3 t^3 - 256000 i_3 t^2 \nonumber \\
 &\qquad + 3200i_3 t - 16i_3 + 729t^2 + 34992 t^2 - 46656 t^5 - 8748 t^3.\nonumber
\end{align}

Replacing our i-invariants into the above system of equations we get:
\begin{align*}
0 &= 86670000\,{t}^{2}-23781600\,t+102789 \\
0 &= -4023934200000\,{t}^{3}+1245222396000\,{t}^{2}-43137816840\,t+73594737\\
0 &= -82315363050000000\,{t}^{5}+61770534511500000\,{t}^{4}-
15443994116835000\,{t}^{3} \\
&\qquad +1287019350200250\,{t}^{2}+106288200\,t- 531441.
\end{align*}
There is only root those three polynomials share: $t= \frac {27}{100}$. Thus, there is exactly one genus 2 curve $C$ defined over
$Q$ (up to $Q$-isomorphism) with a degree 4 elliptic subcover which has an automorphism group $D_{12}$
\[C: \quad y^2=100X^6+100X^3+27\]
Similarly, we show that there are no such curves with automorphism group $D_8$.
\end{proof}


\begin{remark}There are at least  three genus 2 curves defined over $\Q$ with automorphism group $V_4.$
The equations of these curves  are given by the followings:\\

\noindent \textbf{Case 1:} $(i_1,i_2,i_3) =
\left(\frac{66357}{9245},\frac{-892323}{46225},\frac{7776}{459401384375}\right)$

\begin{small}
\begin{align*}
C: &y^2 = 1432139730944\,{x}^{6}+34271993769359360\,{x}^{5} +267643983706245216000\,{x}^{4}\\
&+1267919172426862313120000\,{x}^{3}
+ 23945558970224886213835350000\,{x}^{2} \\
&+ 274330666162649153793599380475000\,x + 1025623291911204380755800513010015625.
\end{align*}
\end{small}

\noindent \textbf{Case 2:} $(i_1,i_2,i_3) = \left(\frac{235629}{1156805},
\frac{-28488591}{214008925},\frac{53747712}{80459143207503125}\right)$

\begin{small}
\begin{align*}
C: &y^2 =  41871441565158964373437321767075023159296\,{x}^{6}\\
&+ 156000358914872008908017177004915818496000\,{x}^{5}\\
&+ 8994429753268252328699175313122263040000000\,{x}^{4}\\
&+ 17857537403821561579480053574533120000000000\,{x}^{3}\\
&+ 775018151562516781352226536816640000000000000\,{x}^{2}\\
&+ 1158249382368691011679236899376000000000000000\,x\\
&+ 26787527679468514273175655200959888458251953125.
\end{align*}
\end{small}

\noindent \textbf{Case 3:} $(i_1,i_2,i_3)
=\left(\frac{1078818669}{383775605},\frac{-77466710644803}{16811290377025},\frac{1356226634181762}{161294078381836186878125}\right)$

\begin{small}
\begin{align*}
C: &y^2 = 9224408124038149308993379217084884661375653227720704\,{x}^{6} \\
&+ 3730758767668984877725129604888152322035364826481920000\,{x}^{5}\\
& + 1138523283803439912403861944281998092255345913017540000000\,{x}^{4}\\
&+ 189425049047781784623261895238590658674841204883457500000000\,{x}^{3}\\
&  + 76212520567614919095032412154382218443932939483817128906250000\,{x}^{2
}\\
&+16717294192073070547056921515101088692898208834624180908203125000\,x\\
&   +2766888989045448736067444316860942956954296161559210811614990234375.
\end{align*}
\end{small}
\end{remark}

We summarize by the following:

\begin{theorem} Let $\psi: C \to E $ be a degree 4 covering of an elliptic curve by a genus 2 curve. Then the following hold:

i) In the generic case the equation of $C$  can be written as follows:
\[
C: y^2 = a_6 x^6+ a_5 x^5 + \dots + a_1 x +a_0
\]
where
\begin{small}
\[
\begin{split}
a_6 =\, & {p}^{2}+b \\
a_5 =\, & 4\,{p}^{3}-6\,{p}^{2}+4\,pb-6\,b \\
a_4 = \,& -4\,{p}^{4}-10\,{p}^{3}+ \left( -5\,b+13 \right) {p}^{2}-8\,pb+12\,b \\
a_3 = \,& 12\,{p}^{4}+ \left( 4+6\,b \right) {p}^{3}+ \left( -12+12\,b \right) { p}^{2}+ \left( 8\,{b}^{2}-6\,b
\right) p-8\,b-8\,{b}^{2} \\
a_2 = \,& \left( -11-4\,b \right) {p}^{4}+ \left( -20\,b+6 \right) {p}^{3}+
 \left( 4+13\,b-12\,{b}^{2} \right) {p}^{2}+10\,pb+12\,{b}^{2} \\
a_1 =\, & \left( 14\,b+2 \right) {p}^{4}+ \left( 6\,{b}^{2}-4+4\,b \right) {p}^ {3}+ \left( -24\,b+6\,{b}^{2} \right)
{p}^{2}+ \left( -6\,{b}^{2}+4\,b
 \right) p-6\,{b}^{2}
 \\
a_0 =\, &\left( -{b}^{2}+1-11\,b \right) {p}^{4}+ \left( 14\,b-2\,{b}^{2}
 \right) {p}^{3}-2\,b{p}^{2}+2\,{b}^{2}p+{b}^{2}.
\end{split}
\]
\end{small}

 ii) In the degenerate case the equation of $\L_4^\prime$ is given by

\begin{tiny}
\begin{equation*}
\begin{split}
 1541086152812576000\,{{ J_2}}^{2}{{ J_4}}^{2}-22835312232360960000 \,{ J_2}\,{J_4}\,{ J_6}+5009676947631\,{{
J_2}}^{6} \\
-8782271900467200000\,{{ J_6}}^{2} + 1176812184652746480\,{{ J_2}}^{4}{ J_4}+12448207102988800000\,{{
J_4}}^{3}\\
-3715799948429529600\,{{ J_2}}^{3}{ J_6} &=  0 \\
 1866265600000000\,{{J_2}}^{2}{{ J_4}}^{4}+ 1389621447673433587445760000000000\,{{J_{10}}}^{2}+282429536481\,{{J_2}}^{10}\\
+6199238007360000\, {{ J_2}}^{6}{{ J_4}}^{2}-256000000000000\,{{ J_4}}^{5}- 2824915237592400\,{{J_2}}^{8}{J_4}\\
 +2665762699498787923200000\,{{J_2}}^{5}{J_{10}}-5102020224000000\, {{J_2}}^{4}{{J_4}}^{3}\\
+6930676241452032000000000\,{ J_2}\,{{ J_4}}^{2}{ J_{10}} + 17635167081823887360000000\,{{J_2}}^{3}{J_4}\,{J_{10}} &= 0 \\
\end{split}
\]
\end{tiny}

 iii) The intersection $\L_4^\prime \cap \M_2 ( D_{8}) = \emptyset$ and the intersection $\L_4^\prime \cap \M_2 ( D_{12})$ contains a single point, namely the curve
 \[C: \quad y^2=100X^6+100X^3+27\]

 \end{theorem}

\end{document}